\documentclass[12pt,reqno]{amsart}

\usepackage{amssymb}
\usepackage{hyperref}
\usepackage{tikz}

\numberwithin{equation}{section}

\newtheorem{theorem}{Theorem}[section]
\newtheorem{proposition}[theorem]{Proposition}

\theoremstyle{definition}

\begin{document}

\baselineskip=15pt

\title[finite dimensional algebras of Vector Fields in $\mathbb{C}^N$]{Lie's
classification of finite dimensional algebras of Vector Fields in $\mathbb{C}^N$}

\author[H. Azad]{Hassan Azad}

\address{Abdus Salam School of Mathematical Sciences, GCU, Lahore 54600, Pakistan}

\email{hassan.azad@sms.edu.pk}

\author[I. Biswas]{Indranil Biswas}

\address{Department of Mathematics, Shiv Nadar University, NH91, Tehsil Dadri,
Greater Noida, Uttar Pradesh 201314, India}

\email{indranil.biswas@snu.edu.in, indranil29@gmail.com}

\author[S. W. Shah]{Said Waqas Shah}

\address{Abdus Salam School of Mathematical Sciences, GCU, Lahore 54600, Pakistan}

\email{waqas.shah@sms.edu.pk}

\date{}

\begin{abstract}
Brief proofs of classical results of Lie on finite dimensional subalgebras
of vector fields in two and three variables are outlined.
The results for algebras of maximal rank for vector fields in
$\mathbb{C}^N$ --- $N$ arbitrary --- are also given.
\end{abstract}

\maketitle

\section{Introduction}

The proofs of Lie's results on classification of finite subalgebras of vector fields
in two and three variables \cite{Lie2} continue
to remain relatively inaccessible.
The main reason seems to be that Lie did not make
use of tools like representation theory and
root systems because these tools were not
available to him. They did become widely known a
few years after Lie published his proofs
but were not used by specialists in the
subject.

If one uses these tools, the proofs
simplify considerably and point a direction
for the classification in higher number of variables.

We will give an outline of the main ideas and
give references to detailed proofs.
The applications of such classification
are too many to be given in this paper.
The interested reader could consult the books
of Ibragimov \cite{Ib1}, \cite{Ib2},
Olver \cite{Ol} and the article by
Ibragimov
``Sophus Lie and Harmony in Mathematical Physics''
\cite{Ib3}.

\section{Basic Definitions}

A vector field defined on an open subset $U$ of
$\mathbb{R}^N$ is a vector valued function
\[
V\ =\ (a_1,\, \cdots,\, a_N).
\]

We assume that the functions $a_i$ are $C^\infty$.
If these functions are real or complex analytic, then
$V$ is called a real analytic or a complex analytic
vector field.

We identify $V$ with the directional derivative
in the direction $V$.
Thus if $V$ is defined on an open subset $U$
and $f$ is a $C^\infty$--function on $U$, then
\[
V(f)\ =\ \sum a_i \frac{\partial f}{\partial x_i}.
\]

If $W$ is another vector field defined on $U$,
then $[V,\, W](f)$ is, by definition,
\[
V(W(f))\, -\, W(V(f)).
\]

With this definition the space of all vector fields on $U$
becomes a Lie algebra.

If $L_1,\, L_2$ are two such algebras, Lie considered
them to be equivalent if by a local change of variables,
one can be transformed into the other.

Thus if $L$ is an abelian algebra of dimension $r$
defined on an open subset $U$ of $\mathbb{R}^N$ and
$X_1,\,\cdots,\,X_r$ is a basis of $L$ and $p$ a point
in $U$ with the tangent vectors $X_1(p),\,\cdots,\,X_r(p)$ linearly independent,
then there is a change of variables
$\widetilde{x}_1,\, \cdots,\widetilde{x}_N$ in which
\[
X_i\ =\ \frac{\partial}{\partial \widetilde{x}_i}
\]
on a neighbourhood of $p$.

Lie's classifications were based on the nature of
generic orbits, whether the algebra was transitive
or intransitive, whether it was primitive or
imprimitive --- meaning that it had an invariant
foliation.

He classified complex analytic finite dimensional
algebras of analytic vector fields in $2$ variables
completely and partially when $N\,=\,3$.

A complete classification is claimed in Amaldi \cite{Am1}, \cite{Am2},
but to date no one has been able to verify this.
Dubrovin \cite{Du} has given a reason why
the classification could not be complete for the
solvable algebras.

It is therefore natural to classify semisimple
and algebras with a proper Levi decomposition
and compare with Lie and Amaldi.

We have attempted this in the papers
\cite{ABMS1}, \cite{ABMS2} and \cite{AABMS}.
Semisimple algebras of vector fields on $\mathbb{C}^N$
of maximal rank have been classified in \cite{ABM2}.

The method in all these classifications is algebraic
and and it is based on the equality of algebraic and
geometric ranks of Cartan subalgebras of semisimple
Lie algebras of vector fields \cite{ABM}.

\subsection{Geometric Rank}

Let $L$ be a Lie algebra of vector fields defined
on an open subset $U$ of $\mathbb{R}^N$.
The geometric rank of $L$ is, by definition,
\[
\max \{ \dim (X(p))\,\,\big\vert\,\,\, X \,\in\, L,\ p \,\in\, U\}.
\]

Thus the rank of the algebra
\[
\left\langle
\frac{\partial}{\partial x}, \ y\frac{\partial}{\partial x},\
\cdots,\ y^{N}\frac{\partial}{\partial x}
\right\rangle
\]
is $1$ while its dimension is $(N+1)$.

If $L$ is a complex semisimple Lie algebra of
vector fields and $C$ is a Cartan subalgebra of $L$,
then a basic fact is that its geometric rank and
dimension coincide \cite{ABM}.

\subsection{Highest Weights}

Equally we need to recall the definition of highest weights.

If $B$ is a Borel subalgebra of a complex semisimple
algebra $L$, and $C$ is a Cartan subalgebra of $L$
contained in $B$, then $B\, =\, C \oplus B'$.

If $\rho\, :\, L\, \longrightarrow\, \mathfrak{gl}(V)$ is a finite dimensional
representation of $L$, a vector $v$ in $V$ is called
a highest weight vector if $\rho(B')(v)\, =\, 0$ and
$v$ is a common eigenvector for $\rho(C)$.

The reader is referred to Kirillov \cite{Ki} for
all undefined terms in this paper.

\section{Applications}

Before giving applications, let us give for completeness, Lie's classification of finite dimensional 
subalgebras of vector fields on the line, up to local equivalence.

\medskip

\noindent\textbf{Sketch:}\
If \( X \,=\, f(x)\,\frac{d}{dx} \) commutes with \( Y \,=\, g(x)\,\frac{d}{dx} \) and \( f \) is not
identically zero, then \( Y \) must be a multiple of \( X \). Therefore if \( L \) is a semisimple algebra
of vector fields on the line, its Cartan subalgebras must be of dimension \(1\) and \( L \) must be
isomorphic to \( \mathfrak{sl}(2,\mathbb{C}) \). Hence \( L \) is generated by \( X,\,Y \) which are
eigenvectors of \( H \,=\, [X,\,Y] \) with nonzero and opposite eigenvalues. Therefore, in the coordinates
in which the Cartan algebra is generated by \( \frac{d}{dx} \), we may suppose that
\[
X\ =\ \exp(x)\,\frac{d}{dx}
\ \ \, \text{ and }\ \ \,
Y\ =\ \exp(-x)\,\frac{d}{dx}
\]
(up to a constant). Moreover as the centralizer of \( X \) is \( X \), up to constants, we see that there 
can be no highest weight vectors in any complement to \( L \) in any extension of \( L \).

\medskip

Now if \( L \) is nilpotent then it must be abelian of dimension \(1\) as a nilpotent algebra has a 
nonzero center.

\medskip

If \( L \) is solvable and not nilpotent, then in the coordinate \( x \) in which the commutator of \( L 
\) is generated by \( \frac{d}{dx} \) we see that if
\[
Z\ =\ f(x)\,\frac{d}{dx}
\]
is in \( L \), then \( f'(x) \) is a constant and therefore \( f(x)\, =\, ax + b \).

\medskip

This completes the local classification of finite dimensional algebras of vector fields on the line.

In the rest of the paper we give applications of the equality of
geometric rank and dimension of Cartan subalgebra to the classification
problem of finite dimensional algebra of vector fields.

\medskip
\textbf{Notation.}\ $A_n$ is the Lie algebra of the Lie group $\text{SL}(n+1,
{\mathbb C})$, $B_n$ is the Lie algebra of the Lie group $\text{SO}(2n+1,
{\mathbb C})$ and $D_n$ is the Lie algebra of the Lie group $\text{SO}(2n,
{\mathbb C})$. The algebra ${\mathfrak g}_2$ is a subalgebra of
$\mathfrak{so}(8, {\mathbb C})$, given as fixed point locus of the automorphism
of order $3$, induced by the graph automorphism of the Dynkin diagram
for $D_4$.

\begin{proposition}\label{prop1}
The only semisimple algebra of vector fields
on $\mathbb{C}^2$ can only be of types
$A_1,\, A_2$ or $A_1 \times A_1$.
\end{proposition}

\begin{proof}
By the equality of the geometric and algebraic ranks
of Cartan subalgebras of semisimple Lie algebras,
the Cartan subalgebra of a semisimple algebra $S$ of vector
fields on $\mathbb{C}^2$ can be of dimension at most $2$.
Thus $S$ can only be of type
$$A_1,\ A_1 \times A_1,\ A_2,\ B_2\ \text{ or }\ {\mathfrak g}_2.$$

The algebra $B_2$ contains an algebra of type
$A_1 \times A_1$ and ${\mathfrak g}_2$ contains an algebra of type $A_2$.
The derived algebra of Borel subalgebras of both these
types contains an abelian subalgebra of rank $2$.
Thus in suitable coordinates this algebra is
\[
\langle \partial_x,\ \partial_y \rangle.
\]

Therefore if $B_2$ could be realized as an algebra of
vector fields on $\mathbb{C}^2$, then we would have
$A_1 \times A_1\ \subset\ B_2$ and $A_1\times A_1$ could
not have an invariant complement.

Similarly $A_2\ \subset\ {\mathfrak g}_2 \ \subset\ V(\mathbb{C}^2)$
shows that $A_2$ could not have an invariant complement
in ${\mathfrak g}_2$.
\end{proof}

\begin{proposition}\label{prop2}
If $L\ =\ S \ltimes R$ is a Levi decomposable algebra
of vector fields on $\mathbb{C}^2$, then $S$ must be
isomorphic to $\mathfrak{sl}(2,\mathbb{C})$.
\end{proposition}

\begin{proof}
If $S$ were of rank $2$, then $S$ must be isomorphic to
$\mathfrak{sl}(2,\mathbb{C})\times \mathfrak{sl}(2,\mathbb{C})$
or $\mathfrak{sl}(3,\mathbb{C})$ and the derived algebras of 
Borel subalgebras would have an abelian algebra of rank $2$.
Therefore, it would have no highest weight vectors in the radical
$\mathcal R$.
\end{proof}

\begin{proposition}\label{prop3}
If $L$ is a semisimple algebra of vector fields
on $\mathbb{C}^N$ and it has a Cartan subalgebra of
dimension $N$, then
\[
L\ \cong\ \bigoplus_{i=1}^{m}
\mathfrak{sl}(k_i+1,\mathbb{C}),
\ \ \, k_1 + \cdots + k_m\ =\ N.
\]
\end{proposition}

A proof of Proposition \ref{prop3} is given in \cite{ABM2}.

\subsection{Classification of semisimple algebras of
vector fields on $\mathbb{C}^3$}\label{se4}

Semisimple algebras of
vector fields on $\mathbb{C}^3$ can only be of rank at most $3$.
Besides algebras of maximal rank, which are of type $A_1\times A_1
\times A_1$, $A_1\times A_2$ or $A_3$, we are left with types
$A_1$, $A_1\times A_1$, $A_2$, $B_2$ and $\mathfrak{g}_2$.

The main idea of canonical forms of these algebras is given in
Section 2 of \cite{AABMS}.

The Lie algebra $A_2$ contains the Heisenberg algebra $[X,\, Y]\ =\ Z$,
while $B_2$ contains $A_1\times A_1$ and $\mathfrak{g}_2$ contains $A_2$.

Using the commutation relations given by the root system of $\mathfrak{g}_2$,
we see that $A_2$ in $V({\mathbb C}^3)$ cannot be extended to $\mathfrak{g}_2$
in $V({\mathbb C}^3)$. The canonical forms of the Heisenberg algebra and
$A_1\times A_1$ together with the commutation relations given by the root
systems are used in determining the fundamental root vectors and in finding all
all the embeddings of rank two algebra in $V({\mathbb C}^3)$.

The reader is referred to \cite{AABMS} for representations of all these
algebras as vector fields in $V({\mathbb C}^3)$. Clearly, if a semisimple
algebra $\mathcal S$ has a representation in which the derived algebra of its
Borel subalgebra has an abelian algebra of rank $3$, then $\mathcal S$
can have no abelian extensions. It is by such considerations that representations
are determined to be inequivalent or not.

\subsection{Classification of Levi--decomposable subalgebras
of vector fields on $\mathbb{C}^3$}

Let $L$ be a Levi--decomposable subalgebra of vector fields on $\mathbb{C}^3$,
say $$L\ =\ S\ltimes {\mathcal R},$$ where $S\, \subset\, V({\mathbb C}^3)$
is semisimple. The semisimple part is known from Section \ref{se4}. If the
derived algebra of a Borel subalgebra of $S$ has an abelian subalgebra of rank
$3$, then arguing as in the proof of Proposition \ref{prop3} (given in
\cite{ABM2}) it follows that $\mathcal R$ has to be zero. Thus $S$ can only be
of types $A_1$, $A_1\times A_1$, $A_2$ and $B_2$.

To find $\mathcal R$ one uses representation theory to build $\mathcal R$ and
we must ensure that $\mathcal R$ is indeed a solvable algebra.

As there is no general theory to find solvable extensions of semisimple algebras,
one has to argue case by case. This is work in progress.

In the preface of \cite{Lie2} Lie writes that he has done all these classifications
for applications in physics. He did not elaborate further on this.

\end{document}